\documentclass[a4paper,11pt]{article}
\usepackage{amsmath,amssymb,amsthm}
\usepackage{mathtools}
\usepackage{microtype}

\usepackage[bookmarks]{hyperref}

\usepackage{tikz}
\tikzstyle{dot}=[draw, fill =black, circle, inner sep=0pt, minimum size=2pt]

\newtheorem{theorem}{Theorem}
\newtheorem{proposition}[theorem]{Proposition}

\theoremstyle{definition}

\newcommand{\tss}[1]{\textsuperscript{#1}}
\newcommand{\ul}[1]{\underline{#1}}

\title{Cellular Structure of Wreath Product Algebras}
\author{\begin{tabular}{c}Reuben Green\thanks{Supported by EPSRC grant [EP/M508068/1]} \\ \texttt{\scriptsize rmg29@kent.ac.uk}\end{tabular}\\[1em]
\small School of Mathematics, Statistics and Actuarial Science,\\
\small University of Kent, CT2 7NF, UK}
\date{}

\begin{document}
\maketitle

\begin{abstract}
{\noindent}
We apply the method of iterated inflations to show that the wreath product of a cellular algebra with a symmetric group is cellular, and obtain
descriptions of the cell and simple modules together with a semisimplicity condition for such wreath products.
\end{abstract}

\section{Introduction}
The wreath product $G{\wr}S_n$ of a finite group $G$ with a symmetric group $S_n$ is a natural group-theoretic construction with many applications;
for example, wreath products $S_m{\wr}S_n$ of two symmetric groups are of great importance in the representation theory of the symmetric group. It is
also natural to consider the wreath product $A{\wr}S_n$ of a $k$-algebra $A$ with a symmetric group $S_n$, see for example the work of Chuang and Tan
in \cite{CHTAN}. The notion of a \textit{cellular algebra} was introduced by Graham and Lehrer in \cite{GLCA} and has since found broad application.
The question arises as to whether a cellular structure on an algebra $A$ yields a cellular structure on the algebra $A{\wr}S_n$, and in \cite{GEGO}
Geetha and Goodman showed that this is so in the case that $A$ is not only cellular but \textit{cyclic cellular}, meaning that all of the cell modules
of $A$ are cyclic \cite[Theorem 4.1]{GEGO}. Their proof is quite combinatorial in nature, and draws on the work of Dipper, James, and Mathas in
\cite{DJM} and of Murphy in \cite{MURHAAN}. However, we shall prove (section \ref{it_inf_awrsn:sec}) that $A{\wr}S_n$ is cellular for any cellular
algebra $A$, by exhibiting it as an \textit{iterated inflation} of tensor products of group algebras of symmetric groups. Iterated inflations were
originally introduced by K\"onig and Xi in \cite{KXIM}, but we shall use this concept in the form given in \cite{GPIICA}. Further, we shall obtain
a convenient graphical description of a well-known method of constructing $A{\wr}S_n$ modules (section \ref{wr_prd_alg:sec}), which we may apply to
the cell and simple modules of $A{\wr}S_n$, and we shall determine a condition for the semisimplicity of $A{\wr}S_n$.

The author is grateful to his Ph.D. supervisor, Rowena Paget, for numerous helpful conversations about this paper and its contents.

\section{Recollections and definitions}
We let $k$ be a field of characteristic $p$ ($p$ may be zero or a prime). By a \textit{$k$-algebra}, we shall mean a finite-dimensional unital
associative $k$-algebra; we shall abbreviate $\otimes_k$ to $\otimes$; all of our modules will be right modules of finite $k$-dimension. By an
\textit{anti-involution} on a $k$-algebra $A$, we mean a self-inverse $k$-linear isomorphism $a \mapsto a^\ast$ such that $(ab)^\ast = b^\ast a^\ast$
for all $a,b \in A$.

For $n$ a non-negative integer, a \emph{composition of $n$} is a tuple of non-negative integers whose sum is $n$, and if $\mu=(\mu_1,\ldots,\mu_t)$ is
a tuple of non-negative integers then we call the numbers $\mu_i$ the \emph{parts} of  $\mu$, and define $|\mu|$ to be the sum $\mu_1+\cdots+\mu_t$, so
that $\mu$ is a composition of $|\mu|$. A composition whose entries are positive and appear in non-increasing order is a \emph{partition}. Note that
$n=0$ has exactly one partition, the empty tuple, which we shall write as $()$.

\subsection{Cellular algebras}
We refer the reader to \cite{GLCA} for basic information and notation on cellular algebras. We shall refer to elements of the poset $\Lambda$ indexing
the cell modules of a cellular algebra as \textit{cell indices}, and we shall write the anti-involution on a cellular algebra $A$ as $a \mapsto a^\ast$.
Recall that to each cell index $\lambda$ we associate a finite set $M(\lambda)$, and we have a \emph{cellular basis} of $A$ whose elements are indexed
by the disjoint union of the sets $M(\lambda) \times M(\lambda)$ for $\lambda \in \Lambda$; we write the cellular basis element indexed by
$(S,T) \in M(\lambda) \times M(\lambda)$ as $C^\lambda_{S,T}$. We call the tuple $(\Lambda,M,C)$ the \emph{cellular data} of $A$ with respect to $\ast$.
Since we are using right modules we take the multiplication rule for cellular basis elements to be
\begin{equation}\label{cell_alg_mult_def:eq}
C^\lambda_{S,T}a \equiv \sum_{X \in M(\lambda)}R_a(T,X)C^\lambda_{S,X}
\end{equation}
modulo cellular basis elements of lower cell index (where $R_a(T,X) \in k$). Then the right cell module $\Delta^\lambda$ is the vector space with basis
$\{C_T : T \in M(\lambda) \}$; our form of the multiplication rule \eqref{cell_alg_mult_def:eq} means that the action of $A$ on $\Delta^\lambda$ is
\begin{equation}\label{cell_mod_def:eq}
C_Ta = \sum_{X \in M(\lambda)}R_a(T,X)C_X.
\end{equation}
 Let us recall some basic
results on
cell modules, see \cite[sections 2 and 3]{GLCA}. Indeed, each cell module is equipped with a bilinear form, whose radical is either the whole cell
module or else its unique maximal $A$-submodule; we shall call these bilinear forms the \textit{cell forms} and their radicals the \emph{cell
radicals}. We let $\Lambda_0$ be the set of $\lambda \in \Lambda$ such that the cell radical of $\Delta^\lambda$ does not equal $\Delta^\lambda$, and
for $\lambda \in \Lambda_0$ we let $L^\lambda$ be the quotient of $\Delta^\lambda$ by its cell radical; thus $L^\lambda$ is a simple $A$-module, and
the modules $L^\lambda$ for $\lambda \in \Lambda_0$ are in fact a complete list of all the right $A$-simples up to isomorphism without redundancy.

\subsection{The symmetric group}
We let $S_n$ denote the symmetric group on the set $\{1,2,\ldots,n\}$, and we take $S_n$ to act on the right, so that the product $\sigma\pi$ of
permutations is calculated by first applying $\sigma$ and then applying $\pi$; thus we write permutations to the right of their arguments. We shall
find it convenient to represent permutations via \textit{permutation diagrams}; for example, we represent $(1,2,3)(5,7) \in S_7$ by the diagram
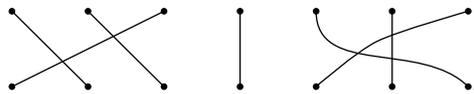
\[\begin{tikzpicture}[line width=0.5pt]
	%top row
	\node[dot] at (1,2){};
	\node[dot] at (2,2){};
	\node[dot] at (3,2){};
	\node[dot] at (4,2){};
	\node[dot] at (5,2){};
	\node[dot] at (6,2){};
	\node[dot] at (7,2){};
	%bottom row
	\node[dot] at (1,1){};
	\node[dot] at (2,1){};
	\node[dot] at (3,1){};
	\node[dot] at (4,1){};
	\node[dot] at (5,1){};
	\node[dot] at (6,1){};
	\node[dot] at (7,1){};
	%lines
	\draw (1,2) to (2,1);
	\draw (2,2) to (3,1);
	\draw (3,2) to (1,1);
	\draw (4,2) to (4,1);
	\draw (5,2) to[out=270] (7,1);
	\draw (6,2) to (6,1);
	\draw plot [smooth] coordinates {(7,2) (5.8,1.6) (5,1)};
\end{tikzpicture},\]
where the $i$\tss{th} \textit{node} on the top row is connected by a \textit{string} to the $(i)\sigma$\tss{th} node on the bottom row.
To calculate the product $\sigma\pi$ in $S_n$ using permutation diagrams, we connect the diagram for $\sigma$ above the diagram for $\pi$, and then
simplify the resulting diagram to yield the permutation diagram of $\sigma\pi$. For $\mu$ a composition of $n$, we write $S_\mu$ for the Young subgroup
of $S_n$ associated to $\mu$.

The group algebra $kS_n$ is known to be a cellular algebra \cite[Theorem 3.20]{MATHBK}, with respect to the anti-involution $\ast$ defined by setting
$\sigma^\ast = \sigma^{-1}$ for $\sigma \in S_n$, and a tuple of cellular data including the partially ordered set $\Lambda_n$
consisting of all partitions of $n$ endowed with the dominance order $\trianglelefteq$. We shall not need the details of the cellular basis occurring in
this structure, but we note that for $\lambda \in \Lambda_n$, the right cell module associated to $\lambda$ by this structure, which we shall denote by
$S^\lambda$, is the (contragradient) dual of the right Specht module defined by James in \cite{JAMSN}\footnote{
See \cite{MATHBK}, ``Warning'' on p.38 and ``Note 2'' on page 54, but note that the original published text incorrectly states that the cell module
obtained is the dual of the right James Specht module associated to the \emph{conjugate} of $\lambda$; see the correction to the Warning in the
author's errata.
}.
Further, the simple modules are indexed by \emph{$p$-restricted partitions}. If $p = 0$ then all partitions are considered $p$-restricted, while if
$p>0$ then a partition is $p$-restricted if the difference between any two consecutive parts is less than $p$. Note that $()$ is $p$-restricted for
all $p\geqslant 0$. For $\lambda$ $p$-restricted, we denote the associated simple module by $D^\lambda$ \cite[Theorem 3.43]{MATHBK}.

The following result may easily be proved by directly verifying the axioms for a cellular algebra; in fact, it is merely a special case of the
general result that a tensor product of cellular algebras is cellular, see for example section 3.2 of \cite{GEGO}.

\begin{proposition}\label{ten_prod_sn:prop}
Let $n_1,\ldots,n_r$ be non-negative integers. Then the group algebra $k(S_{n_1}\times\cdots\times S_{n_r})$ is a cellular algebra  with respect to
the map given by $(\sigma_1,\ldots,\sigma_r) \longmapsto (\sigma^{-1}_1,\ldots,\sigma^{-1}_r)$ for $\sigma_i \in S_{n_i}$ and a cellular structure
where the poset of cell indices is $\Lambda_{n_1}\times\cdots\times\Lambda_{n_r}$ with the order where
$(\lambda_1,\ldots,\lambda_r) \geqslant (\mu_1,\ldots,\mu_r)$ means  $\lambda_i \trianglerighteq \mu_i$ for all $i$; the cell module associated to
$(\lambda_1,\ldots,\lambda_r)$ is $S^{\lambda_1}\otimes\cdots\otimes S^{\lambda_r}$ with the action
\[
(x_1 \otimes\cdots\otimes x_r)\cdot(\sigma_1,\ldots,\sigma_r) = (x_1\sigma_1)\otimes\cdots\otimes(x_r\sigma_r)
\]
for $x_i \in S^{\lambda_i}$, $\sigma_i \in S_{n_i}$, and the cell form is given on pure tensors by
\[ \langle x_1\otimes\cdots\otimes x_r,y_1\otimes\cdots\otimes y_r\rangle = \langle x_1,y_1 \rangle\cdots\langle x_r,y_r \rangle \]
where each bilinear form on the right hand side is the appropriate cell form of some $S^{\lambda_i}$.
\end{proposition}

Let $\sigma \in S_n$. Then an \emph{inversion} of $\sigma$ is a transposition $(i,j)$ in $S_n$ such that $1 \leq i < j \leq n$ but 
$(i)\sigma > (j)\sigma$, and
the \emph{Coxeter length} of $\sigma$ is defined to be the total number of inversions of $\sigma$; we shall simply call this the \textit{length} of
$\sigma$. It is well-known that
if $\mu$ is a composition of $n$, then each right coset $S_\mu\sigma$ of $S_\mu$ contains a unique element of minimal length, and further that if
$\mu = (\mu_1,\ldots,\mu_r)$, then for any given right $S_\mu$-coset, the element of minimal length is the unique element $\gamma$ of the coset
such that in the sequence $(1)\gamma^{-1},\ldots,(n)\gamma^{-1}$, the elements $1,\ldots,\mu_1$ occur in increasing order, as do the elements
$\mu_1+1,\ldots,\mu_1+\mu_2$, the elements $\mu_1+\mu_2+1,\ldots,\mu_1+\mu_2+\mu_3$, and so on. Equivalently, an element $\sigma$ of $S_n$ is
of minimal length in its coset $S_\mu\sigma$ if and only if, in its permutation diagram, the strings attached to the first $\mu_1$ nodes on the top row
do not cross each other, the strings attached to the next $\mu_2$ nodes on the top row do not cross each other, and so on. For example, the permutation
whose diagram appears in the diagram \eqref{mod_diag_eg:diag} below is of minimal length in its $S_\mu$ coset for $\mu=(3,2,3)$. For any $\mu$ a
composition of $n$, we define $\mathcal{R}_\mu$ to be the unique system of minimal-length right $S_\mu$-coset representatives in $S_n$.

\subsection{Iterated inflation of cellular algebras}\label{it_inf:sec}
Iterated inflations of cellular algebras were first introduced by K\"onig and Xi in \cite{KXIM}, but we shall use them as presented in \cite{GPIICA}.
We shall now summarise the content of \cite{GPIICA}; note however that we give the form using right cell modules, rather than the left cell modules
used in \cite{GPIICA}.

Let $A$ be a $k$-algebra, with an anti-involution $\ast$. Suppose that we have, up to isomorphism of $k$-vector spaces, a decomposition
\[ A \cong \bigoplus_{\mu \in I}V_\mu  \otimes B_\mu  \otimes V_\mu  \]
of $A$, where $I$ is a finite partially ordered set, each $V_\mu $ is a $k$-vector space, and each $B_\mu $ is a cellular algebra over $k$
with respect to an anti-involution $\ast$ and cellular data $\left(\Lambda_\mu ,M_\mu ,C\right)$. We shall henceforth consider $A$ to be identified with
this direct sum of tensor products, and we shall speak of the subspace $V_\mu  \otimes B_\mu  \otimes V_\mu $ as the $\mu $-th \textit{layer} of $A$.
Suppose that for each $\mu  \in I$, we have a basis $\mathcal{V}_\mu $ for $V_\mu $ and a basis $\mathcal{B}_\mu $ for $B_\mu $. Let $\mathcal{A}$ be
the basis of $A$ consisting of all elements $u \otimes b \otimes w$ for all $u,w \in \mathcal{V}_\mu $ and all $b \in \mathcal{B}_\mu $, as $\mu$
ranges over $I$. Suppose that for each $\mu \in I$, we have for any $u,w \in \mathcal{V}_\mu $ and any $b \in \mathcal{B}_\mu $ that
\begin{equation} \label{it_inf_anti_gen:eq}
(u \otimes b \otimes w)^\ast = w \otimes b^\ast \otimes u,
\end{equation}
and suppose further that for any $\mu \in I$ we have maps $\phi_\mu : \mathcal{V}_\mu  \times \mathcal{A} \rightarrow V_\mu $ and
$\theta_\mu : \mathcal{V}_\mu  \times \mathcal{A} \rightarrow B_\mu $ such that for any $u,w \in \mathcal{V}_\mu $ and any $b \in \mathcal{B}_\mu $, we
have for any $a \in \mathcal{A}$ that
\begin{equation} \label{it_inf_mult_gen:eq}
(u \otimes b \otimes w) \cdot a \equiv u \otimes b\,\theta_\mu(w,a) \otimes \phi_\mu(w,a) \quad \mathrm{mod\;} J(<\mu ),
\end{equation}
where $J(<\mu ) = \bigoplus_{\alpha < \mu }V_\alpha  \otimes B_\alpha  \otimes V_\alpha $. Then by \cite[Theorem 1]{GPIICA}, $A$ is cellular with
respect to $\ast$ and the cellular data $(\Lambda,M,C)$, where $\Lambda$ is the set
$\{(\mu,\lambda) : \mu \in I \text{ and } \lambda \in \Lambda_\mu \}$ with the lexicographic order,
$M(\mu,\lambda)$ is $\mathcal{V}_\mu \times M_\mu(\lambda)$, and $C^{(\mu,\lambda)}_{(x,X),(y,Y)} = x \otimes C^{\lambda}_{X,Y} \otimes y$.

Further by \cite[Proposition 2]{GPIICA}, for each $\mu \in I$ there is a unique $B_\mu$-valued $k$-bilinear form $\psi_\mu$ on $V_\mu$ such that for any
$u,w,x,y \in V_\mu$ and $b,c \in B_\mu$ we have $\psi_\mu(y,u) = \psi_\mu(u,y)^\ast$ and
\begin{equation}\label{mult_in_a_layer:eq}
	(x \otimes c \otimes y)(u \otimes b \otimes w) \equiv x \otimes c\,\psi_\mu(y,u) b \otimes w   \quad \mathrm{mod\;} J(<\mu).
\end{equation}

Finally (see \cite[Proposition 3]{GPIICA}), let $(\mu,\lambda) \in \Lambda$, and let $\Delta^\lambda$ be the right cell module of $B_i$ corresponding to
$\lambda$. The right cell module $\Delta^{(\mu,\lambda)}$ of $A$ may be obtained by equipping $\Delta^\lambda \otimes V_\mu$ with the action given, for
$a \in \mathcal{A}$, $x \in \mathcal{V}_\mu$ and $z \in \Delta^\lambda$, by $(z \otimes x)a = z\,\theta_\mu(x,a) \otimes \phi_\mu(x,a)$.
Moreover, if $\langle\cdot\,,\cdot\rangle$ is the cell form on $\Delta^\lambda \otimes V_\mu$ and $\langle\cdot\,,\cdot\rangle_\lambda$ is the
cell form on $\Delta^\lambda$, then for any $x,y \in V_\mu$ and any $z,v \in \Delta^\lambda$, we have
\begin{equation}\label{bilin_form:eq}
	\langle z \otimes x,\, v \otimes y \rangle = \langle z\,\psi_\mu(x,y),v \rangle_\lambda = \langle z,v\,\psi_\mu(y,x) \rangle_\lambda.
\end{equation}

\section{Wreath product algebras}\label{wr_prd_alg:sec}
We recall the notion of the wreath product of an algebra with a symmetric group from \cite{CHTAN}.
Indeed, let $A$ be a finite-dimensional unital associative $k$-algebra. Consider the $k$-vector space $kS_n\otimes A^{\otimes n}$, and further let us
write a pure tensor $x\otimes a_1\otimes a_2\otimes\cdots\otimes a_n$ in this vector space as $(x\,; a_1, a_2, \ldots, a_n)$. Then we have a
well-defined multiplication which is given by
\begin{multline*}
	(\sigma; a_1,a_2,\ldots,a_n)(\pi; b_1,b_2,\ldots,b_n) = \\
	(\sigma\pi; a_{(1)\pi^{-1}}b_1,a_{(2)\pi^{-1}}b_2,\ldots,a_{(n)\pi^{-1}}b_n)
\end{multline*}
for $\sigma,\pi \in S_n$ and $a_i,b_i \in A$; we define the \emph{wreath product} $A{\wr}S_n$ of $A$ and $S_n$ to be the unital associative
$k$-algebra so obtained.

We assume that the reader is familiar with the notion of \emph{diagram algebras}, for example the Brauer or Temperley-Lieb algebras. We can consider
$A{\wr}S_n$ to be a kind of diagram algebra. Indeed, we may represent a pure tensor
$(\sigma; a_1,a_2,\ldots,a_n)$ in $A{\wr}S_n$, where $\sigma \in S_n$ and $a_i \in A$, by a diagram obtained by drawing the permutation diagram
associated to $\sigma$, with the nodes of the bottom row replaced by the elements $a_i$. For example, if $n = 5$ and $\sigma = (1,4,3,5,2)$, then we
represent the element $(\sigma; a_1,a_2,a_3,a_4,a_5)$ by
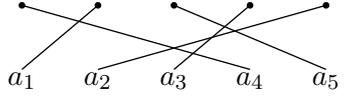
\[\begin{tikzpicture}[line width=0.5pt]
	%top row
	\node[dot] at (1,2){};
	\node[dot] at (2,2){};
	\node[dot] at (3,2){};
	\node[dot] at (4,2){};
	\node[dot] at (5,2){};
	%bottom row
	\node at (1,1){$a_1$};
	\node at (2,1){$a_2$};
	\node at (3,1){$a_3$};
	\node at (4,1){$a_4$};
	\node at (5,1){$a_5$};
	%lines "top to bottom"
	\draw (1,2) to (4,1.15);
	\draw (2,2) to (1,1.15);
	\draw (3,2) to (5,1.15);
	\draw (4,2) to (3,1.15);
	\draw (5,2) to (2,1.15);
	%\draw plot [smooth] coordinates {(7,2) (5.8,1.6) (5,1)};
\end{tikzpicture}.\]
Such diagrams are useful for computing products, as we now show by an example. Indeed, keep $n = 5$ and $\sigma = (1,4,3,5,2)$,
and let $\pi = (1,3,5)(2,4)$. Then to compute the product
$(\sigma; a_1,a_2,a_3,a_4,a_5)(\pi; b_1,b_2,b_3,b_4,b_5)$,
we draw the diagram corresponding to the first factor above the one corresponding to the second factor, to obtain
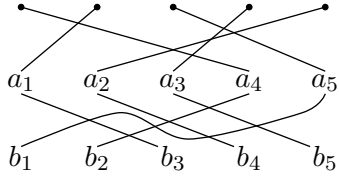
\[\begin{tikzpicture}[line width=0.5pt]
	%top row
	\node[dot] at (1,3){};
	\node[dot] at (2,3){};
	\node[dot] at (3,3){};
	\node[dot] at (4,3){};
	\node[dot] at (5,3){};
	%middle row
	\node at (1,2){$a_1$};
	\node at (2,2){$a_2$};
	\node at (3,2){$a_3$};
	\node at (4,2){$a_4$};
	\node at (5,2){$a_5$};
	%bottom row
	\node at (1,1){$b_1$};
	\node at (2,1){$b_2$};
	\node at (3,1){$b_3$};
	\node at (4,1){$b_4$};
	\node at (5,1){$b_5$};
	%lines "top to middle"
	\draw (1,3) to (4,2.15);
	\draw (2,3) to (1,2.15);
	\draw (3,3) to (5,2.15);
	\draw (4,3) to (3,2.15);
	\draw (5,3) to (2,2.15);
	%lines "middle to bottom"
	\draw (1,1.85) to (2.8,1.15);
	\draw (2,1.85) to (3.8,1.15);
	\draw (3,1.85) to (4.8,1.15);
	\draw (4,1.85) to (2,1.15);
	\draw plot [smooth] coordinates {(5,1.85) (4.7,1.6) (3.2,1.25) (2.3,1.6) (1,1.15)};
\end{tikzpicture}\]
and we then slide each $a_i$ down its string to meet some $b_j$, and then resolve the two connected permutation diagrams into a single
diagram, to obtain
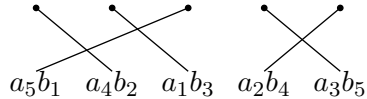
\[\begin{tikzpicture}[line width=0.5pt]
	%top row
	\node[dot] at (1,2){};
	\node[dot] at (2,2){};
	\node[dot] at (3,2){};
	\node[dot] at (4,2){};
	\node[dot] at (5,2){};
	%bottom row
	\node at (1,1){$a_5b_1$};
	\node at (2,1){$a_4b_2$};
	\node at (3,1){$a_1b_3$};
	\node at (4,1){$a_2b_4$};
	\node at (5,1){$a_3b_5$};
	%lines "top to bottom"
	\draw (1,2) to (2,1.16);
	\draw (2,2) to (3,1.16);
	\draw (3,2) to (1,1.16);
	\draw (4,2) to (5,1.16);
	\draw (5,2) to (4,1.16);
\end{tikzpicture}\]
which corresponds to the element $\bigl((1,2,3)(4,5); a_5b_1,a_4b_2,a_1b_3,a_2b_4,a_3b_5\bigr)$, which is indeed the product of the two
elements we started with.

Note that, unlike the usual diagram basis of the Brauer or Temperley-Lieb algebras, the set of all such diagrams is not a basis of $A{\wr}S_n$. A basis
of such diagrams can be formed by fixing a basis $\mathcal{C}$ of $A$, and then taking the set of all elements $(\sigma; a_1, \ldots, a_n)$ for
$\sigma \in S_n$ and $a_i \in \mathcal{C}$; however the product of two such basis elements will not in general be a scalar multiple of another basis
element as is the case for the diagram basis of the Brauer or Temperley-Lieb algebras.

It is easy to show that there is a well-defined anti-involution $\ast$ on $A{\wr}S_n$ given by
\begin{equation}\label{anti_inv_def:eq}
(\sigma;a_1,\ldots,a_n)^\ast = \bigl(\sigma^{-1}\,;\,a_{(1)\sigma}^\ast,\ldots,a_{(n)\sigma}^\ast \bigr),
\end{equation}
where $\sigma \in S_n$ and $a_1,\ldots,a_n \in A$. In terms of diagrams, this map corresponds to the operation of taking a diagram, flipping it about
the horizontal line half-way between its two rows of nodes (so that the elements $a_i$ lie on the \emph{top} row), replacing each element $a_i$ with
its image $a_i^\ast$ under the anti-involution on $A$, and then sliding each element $a_i^\ast$ to the bottom of its string.

Now there is a standard method of constructing modules for $A{\wr}S_n$ from $A$-modules and symmetric group modules; see for example Section 3 of
\cite{CHTAN}. Indeed, let $\mu$ be an $r$-part composition of $n$, $X_1,\ldots,X_r$ be $A$-modules, and for each $i=1,\ldots,r$ let $Y_i$ be a
$kS_{\mu_i}$ module. We write $A{\wr}S_\mu$ for the subalgebra of $A{\wr}S_n$ spanned by all elements $(\sigma;a_1,\ldots,a_n)$ where $a_i \in A$ and
$\sigma \in S_\mu$. Then $X^{\otimes \mu_1}_1\otimes\cdots\otimes X^{\otimes \mu_r}_r\otimes Y_1 \otimes \cdots \otimes Y_r$
is naturally a $A{\wr}S_\mu$-module via the action
\begin{multline*}
(x_1 \otimes \cdots \otimes x_n \otimes y_1 \otimes \cdots \otimes y_r)(\sigma;a_1,\ldots,a_n) = \\
x_{(1)\sigma^{-1}}a_1  \otimes \cdots \otimes x_{(n)\sigma^{-1}}a_n\otimes y_1\sigma_1 \otimes \cdots \otimes y_r\sigma_r,
\end{multline*}
where the elements $\sigma_i \in S_{\mu_i}$ are such that under the natural identification of $S_\mu$ with ${S_{\mu_1}\times\cdots\times S_{\mu_r}}$,
$\sigma$ is identified with $(\sigma_1,\ldots,\sigma_r)$. Then inducing from $A{\wr}S_\mu$ to $A{\wr}S_n$ (that is, applying the functor
$-\otimes_{A{\wr}S_\mu}A{\wr}S_n$) yields a module which we may easily see is isomorphic as a $k$-vector space to
\begin{equation}\label{vsp_cell_mod_std:eq}
X^{\otimes \mu_1}_1\otimes\cdots\otimes X^{\otimes \mu_r}_r\otimes Y_1 \otimes \cdots \otimes Y_r\otimes k\mathcal{R}_\mu,
\end{equation}
where $k\mathcal{R}_\mu$ is the vector space on the basis $\mathcal{R}_\mu$ of minimal-length coset representatives, with the action given by
\begin{multline}\label{wr_mod_full_action:eq}
(x_1 \otimes \cdots \otimes x_n \otimes y_1 \otimes \cdots \otimes y_r \otimes \gamma)(\sigma;a_1,\ldots,a_n) = \\
x_{(1)\theta^{-1}}a_{(1)\zeta} \otimes \cdots \otimes x_{(n)\theta^{-1}}a_{(n)\zeta}\otimes
y_1\theta_1 \otimes \cdots \otimes y_r\theta_r \otimes \zeta,
\end{multline}
where $\gamma \in \mathcal{R}_\mu$, and $\zeta \in \mathcal{R}_\mu$ and $\theta \in S_\mu$ are such that $\gamma\sigma = \theta\zeta$. Letting
$\ul{X}$ be the tuple $(X_1,\ldots,X_r)$ and $\ul{Y}$ be the tuple $(Y_1,\ldots,Y_r)$, we denote the module so obtained by $\Theta^\mu(\ul{X},\ul{Y})$.

We now introduce a diagrammatic representation for certain pure tensors in the module $\Theta^\mu(\ul{X},\ul{Y})$ which provides a very convenient and
intuitive understanding of the action of $A{\wr}S_n$. Indeed, let us take a pure tensor
$x_1 \otimes \cdots \otimes x_n \otimes y_1 \otimes \cdots \otimes y_r \otimes \gamma$
in \eqref{vsp_cell_mod_std:eq}, where $\gamma \in \mathcal{R}_\mu$. We represent this element by taking the permutation diagram of $\gamma$, labelling
the nodes on its lower row from left to right with the elements $x_{(1)\gamma^{-1}},\ldots,x_{(n)\gamma^{-1}}$, then linking together the first $\mu_1$
nodes on the top
row and labelling them with $y_1$, linking together the next $\mu_2$ nodes on the top row and labelling the linked nodes with $y_2$, and so on.
For example, take $n=8$, $r=3$, $\mu = (3,2,3)$, and $\gamma = (2,3,6)(5,8,7)$ ($\gamma$ may be seen to be an element of $\mathcal{R}_\mu$ from its
permutation diagram in \eqref{mod_diag_eg:diag}, since the strings associated to each $y_i$ do not cross each other). We then represent the element
\[x_1 \otimes x_2 \otimes x_3 \otimes x_4 \otimes x_5 \otimes x_6 \otimes x_7 \otimes x_8 \otimes y_1 \otimes y_2  \otimes y_3 \otimes \gamma\]
by the diagram
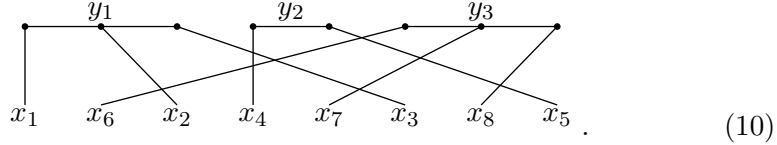
\begin{equation}\label{mod_diag_eg:diag}\begin{tikzpicture}[line width=0.5pt]
	%above top row
	\node at (2,2.4){$y_1$};
	\node at (4.5,2.4){$y_2$};
	\node at (7,2.4){$y_3$};
	%top row
	\draw (1,2.2) to (3,2.2);
		\node[dot] at (1,2.2){};
		\node[dot] at (2,2.2){};
		\node[dot] at (3,2.2){};
	\draw (4,2.2) to (5,2.2);
		\node[dot] at (4,2.2){};
		\node[dot] at (5,2.2){};
	\draw (6,2.2) to (8,2.2);
		\node[dot] at (6,2.2){};
		\node[dot] at (7,2.2){};
		\node[dot] at (8,2.2){};
	%bottom row
	\node at (1,1){$x_1$};
	\node at (2,1){$x_6$};
	\node at (3,1){$x_2$};
	\node at (4,1){$x_4$};
	\node at (5,1){$x_7$};
	\node at (6,1){$x_3$};
	\node at (7,1){$x_8$};
	\node at (8,1){$x_5$};
	%lines top to bottom
	\draw (1,2.2) to (1,1.16);
	\draw (2,2.2) to (3,1.16);
	\draw (3,2.2) to (6,1.16);
	\draw (4,2.2) to (4,1.16);
	\draw (5,2.2) to (8,1.16);
	\draw (6,2.2) to (2,1.16);
	\draw (7,2.2) to (5,1.16);
	\draw (8,2.2) to (7,1.16);
\end{tikzpicture}.\end{equation}
Note that each $x_i$ is connected to the $i$\tss{th} node on the top row. Note also that for each $i=1,2,3$, the elements of $X_i$
are attached to the strings associated to $y_i$.  We thus identify $\Theta^\mu(\ul{X},\ul{Y})$ with the $k$-vector space spanned by diagrams
consisting of the permutation diagram of some element of $\mathcal{R}_\mu$ where (as in \eqref{mod_diag_eg:diag}) for each $i=1,\ldots,r$, the
$(\mu_1+\cdots+\mu_{i-1}+1)$\tss{th} to $(\mu_1+\cdots+\mu_i)$\tss{th} nodes are connected to form a single block which is labelled by an element of
$Y_i$, and where each node on the bottom row is replaced with an element of some $X_j$ such that each top-row node in the $i$\tss{th} block is
connected to an element of $X_i$ on the bottom row. We note that under this identification, the diagram in $\Theta^\mu(\ul{X},\ul{Y})$ whose top row
has labels $y_1$ to $y_r$, whose bottom row has labels $u_1$ to $u_n$, and whose underlying permutation diagram is that of
$\gamma \in \mathcal{R}_\mu$ represents the pure tensor
$u_{(1)\gamma}\otimes\cdots\otimes u_{(n)\gamma}\otimes y_1\otimes\cdots\otimes y_r \otimes\gamma$.
Further note that the set of all such diagrams is not linearly independent in
$\Theta^\mu(\ul{X},\ul{Y})$, and so they form a spanning set rather than a basis.

This diagram representation of $\Theta^\mu(\ul{X},\ul{Y})$ affords an intuitive realisation of the action of $A{\wr}S_n$, and we illustrate this by an
example. Indeed, keeping $n=8$, $r=3$, $\mu = (3,2,3)$ as above, let us consider the diagram
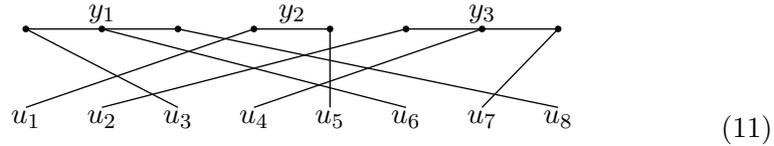
\begin{equation}\label{elt_of_Theta_mu_X_Y:diag}\begin{tikzpicture}[line width=0.5pt]
	%above top row
	\node at (2,2.4){$y_1$};
	\node at (4.5,2.4){$y_2$};
	\node at (7,2.4){$y_3$};
	%top row
	\draw (1,2.2) to (3,2.2);
		\node[dot] at (1,2.2){};
		\node[dot] at (2,2.2){};
		\node[dot] at (3,2.2){};
	\draw (4,2.2) to (5,2.2);
		\node[dot] at (4,2.2){};
		\node[dot] at (5,2.2){};
	\draw (6,2.2) to (8,2.2);
		\node[dot] at (6,2.2){};
		\node[dot] at (7,2.2){};
		\node[dot] at (8,2.2){};
	%bottom row
	\node at (1,1){$u_1$};
	\node at (2,1){$u_2$};
	\node at (3,1){$u_3$};
	\node at (4,1){$u_4$};
	\node at (5,1){$u_5$};
	\node at (6,1){$u_6$};
	\node at (7,1){$u_7$};
	\node at (8,1){$u_8$};
	%lines top to bottom
	\draw (1,2.2) to (3,1.16);
	\draw (2,2.2) to (6,1.16);
	\draw (3,2.2) to (8,1.16);
	\draw (4,2.2) to (1,1.16);
	\draw (5,2.2) to (5,1.16);
	\draw (6,2.2) to (2,1.16);
	\draw (7,2.2) to (4,1.16);
	\draw (8,2.2) to (7,1.16);
\end{tikzpicture}\end{equation}
in $\Theta^\mu(\ul{X},\ul{Y})$; note that this diagram represents the pure tensor
\begin{multline}\label{diag_eg_elt_of_diag_module:eq}
u_3 \otimes u_6 \otimes u_8 \otimes u_1 \otimes u_5 \otimes u_2 \otimes u_4 \otimes u_7 \otimes\\
 y_1 \otimes y_2  \otimes y_3 \otimes (1,3,8,7,4)(2,6).\qquad
\end{multline}
Now take the element
\begin{equation}\label{diag_eg_elt_of_A_wr_Sn:eq}
\bigl((1,2,3)(4,6,8,7,5);a_1,a_2,a_3,a_4,a_5,a_6,a_7,a_8\bigr)
\end{equation}
of $A{\wr}S_8$, which is represented by the diagram
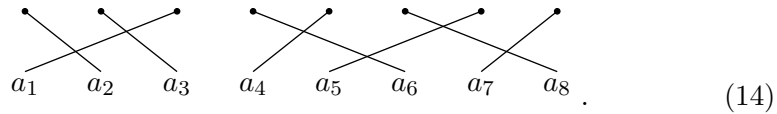
\begin{equation}\label{elt_of_awrsn_acting:diag}\begin{tikzpicture}[line width=0.5pt]
	%top row
	\node[dot] at (1,2){};
	\node[dot] at (2,2){};
	\node[dot] at (3,2){};
	\node[dot] at (4,2){};
	\node[dot] at (5,2){};
	\node[dot] at (6,2){};
	\node[dot] at (7,2){};
	\node[dot] at (8,2){};
	%bottom row
	\node at (1,1){$a_1$};
	\node at (2,1){$a_2$};
	\node at (3,1){$a_3$};
	\node at (4,1){$a_4$};
	\node at (5,1){$a_5$};
	\node at (6,1){$a_6$};
	\node at (7,1){$a_7$};
	\node at (8,1){$a_8$};
	%lines "top to bottom"
	\draw (1,2) to (2,1.2);
	\draw (2,2) to (3,1.2);
	\draw (3,2) to (1,1.2);
	\draw (4,2) to (6,1.2);
	\draw (5,2) to (4,1.2);
	\draw (6,2) to (8,1.2);
	\draw (7,2) to (5,1.2);
	\draw (8,2) to (7,1.2);
\end{tikzpicture}.\end{equation}
The action of the element \eqref{elt_of_awrsn_acting:diag} on \eqref{elt_of_Theta_mu_X_Y:diag} is calculated as follows: we connect the diagram
\eqref{elt_of_awrsn_acting:diag} below the diagram \eqref{elt_of_Theta_mu_X_Y:diag} to get
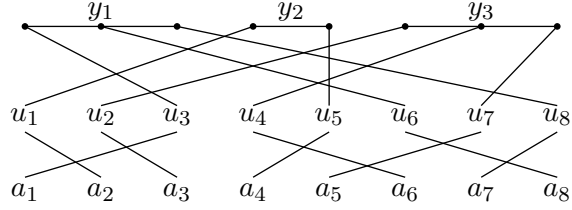
\[\begin{tikzpicture}[line width=0.5pt]
	%above top row
	\node at (2,3.4){$y_1$};
	\node at (4.5,3.4){$y_2$};
	\node at (7,3.4){$y_3$};
	%top row
	\draw (1,3.2) to (3,3.2);
		\node[dot] at (1,3.2){};
		\node[dot] at (2,3.2){};
		\node[dot] at (3,3.2){};
	\draw (4,3.2) to (5,3.2);
		\node[dot] at (4,3.2){};
		\node[dot] at (5,3.2){};
	\draw (6,3.2) to (8,3.2);
		\node[dot] at (6,3.2){};
		\node[dot] at (7,3.2){};
		\node[dot] at (8,3.2){};
	%middle row
	\node at (1,2){$u_1$};
	\node at (2,2){$u_2$};
	\node at (3,2){$u_3$};
	\node at (4,2){$u_4$};
	\node at (5,2){$u_5$};
	\node at (6,2){$u_6$};
	\node at (7,2){$u_7$};
	\node at (8,2){$u_8$};
	%lines top to middle
	\draw (1,3.2) to (3,2.15);
	\draw (2,3.2) to (6,2.15);
	\draw (3,3.2) to (8,2.15);
	\draw (4,3.2) to (1,2.15);
	\draw (5,3.2) to (5,2.15);
	\draw (6,3.2) to (2,2.15);
	\draw (7,3.2) to (4,2.15);
	\draw (8,3.2) to (7,2.15);
	%bottom row
	\node at (1,1){$a_1$};
	\node at (2,1){$a_2$};
	\node at (3,1){$a_3$};
	\node at (4,1){$a_4$};
	\node at (5,1){$a_5$};
	\node at (6,1){$a_6$};
	\node at (7,1){$a_7$};
	\node at (8,1){$a_8$};
	%lines "middle to bottom"
	\draw (1,1.8) to (2,1.2);
	\draw (2,1.8) to (3,1.2);
	\draw (3,1.8) to (1,1.2);
	\draw (4,1.8) to (6,1.2);
	\draw (5,1.8) to (4,1.2);
	\draw (6,1.8) to (8,1.2);
	\draw (7,1.8) to (5,1.2);
	\draw (8,1.8) to (7,1.2);
\end{tikzpicture}.\]
We slide each $u_i$ down its string and simplify the drawing of the resulting partition diagram, to obtain
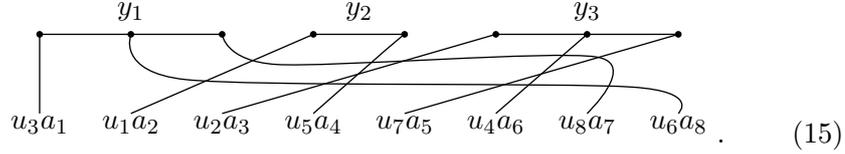
\begin{equation}\begin{tikzpicture}[line width=0.5pt,xscale=1.2]\label{eg_cell_diag_unfact:diag}
	%top row
	\node at (2,2.5){$y_1$};
	\node at (4.5,2.5){$y_2$};
	\node at (7,2.5){$y_3$};
	%lines and nodes for top row
	\draw (1,2.2) to (3,2.2);
		\node[dot] at (1,2.2){};
		\node[dot] at (2,2.2){};
		\node[dot] at (3,2.2){};
	\draw (4,2.2) to (5,2.2);
		\node[dot] at (4,2.2){};
		\node[dot] at (5,2.2){};
	\draw (6,2.2) to (8,2.2);
		\node[dot] at (6,2.2){};
		\node[dot] at (7,2.2){};
		\node[dot] at (8,2.2){};
	%bottom row
	\node at (1,1){$u_3 a_1$};
	\node at (2,1){$u_1 a_2$};
	\node at (3,1){$u_2 a_3$};
	\node at (4,1){$u_5 a_4$};
	\node at (5,1){$u_7 a_5$};
	\node at (6,1){$u_4 a_6$};
	\node at (7,1){$u_8 a_7$};
	\node at (8,1){$u_6 a_8$};
	%lines top to bottom
	\draw (1,2.2) to (1,1.15);
	\draw plot [smooth] coordinates {(2,2.2) (2.6,1.6) (7.5,1.5) (8,1.15)};
	\draw plot [smooth] coordinates {(3,2.2) (3.6,1.8) (7.1,1.9) (7,1.15)};
	\draw (4,2.2) to (2,1.15);
	\draw plot [smooth] coordinates {(5,2.2) (4.8,2) (4,1.15)};
	\draw (6,2.2) to (3,1.15);
	\draw (7,2.2) to (6,1.15);
	\draw (8,2.2) to (5,1.15);
\end{tikzpicture}.\end{equation}
The permutation encoded in the strings of this diagram is $(2,8,5,4)(3,7,6)$, which has the factorisation
$(2,8,5,4)(3,7,6) = (2,3)(7,8)\cdot(2,7,5,4)(3,8,6)$
where $(2,3)(7,8) \in S_\mu$ and $(2,7,5,4)(3,8,6) \in \mathcal{R}_\mu$; we represent this factorisation by redrawing the diagram
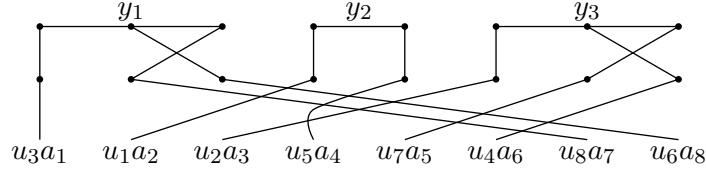
\eqref{eg_cell_diag_unfact:diag} as
\[\begin{tikzpicture}[line width=0.5pt,xscale=1.2]
	%top row
	\node at (2,2.9){$y_1$};
	\node at (4.5,2.9){$y_2$};
	\node at (7,2.9){$y_3$};
	%lines and nodes for top row
	\draw (1,2.7) to (3,2.7);
		\node[dot] at (1,2.7){};
		\node[dot] at (2,2.7){};
		\node[dot] at (3,2.7){};
	\draw (4,2.7) to (5,2.7);
		\node[dot] at (4,2.7){};
		\node[dot] at (5,2.7){};
	\draw (6,2.7) to (8,2.7);
		\node[dot] at (6,2.7){};
		\node[dot] at (7,2.7){};
		\node[dot] at (8,2.7){};
	%lines top row to middle
	\draw (1,2.7) to (1,2);
	\draw (2,2.7) to (3,2);
	\draw (3,2.7) to (2,2);
	\draw (4,2.7) to (4,2);
	\draw (5,2.7) to (5,2);
	\draw (6,2.7) to (6,2);
	\draw (7,2.7) to (8,2);
	\draw (8,2.7) to (7,2);
	%nodes for middle row
	\node[dot] at (1,2){};
	\node[dot] at (2,2){};
	\node[dot] at (3,2){};
	\node[dot] at (4,2){};
	\node[dot] at (5,2){};
	\node[dot] at (6,2){};
	\node[dot] at (7,2){};
	\node[dot] at (8,2){};
	%bottom row
	\node at (1,1){$u_3 a_1$};
	\node at (2,1){$u_1 a_2$};
	\node at (3,1){$u_2 a_3$};
	\node at (4,1){$u_5 a_4$};
	\node at (5,1){$u_7 a_5$};
	\node at (6,1){$u_4 a_6$};
	\node at (7,1){$u_8 a_7$};
	\node at (8,1){$u_6 a_8$};
	%lines "middle to bottom"
	\draw (1,2) to (1,1.2);
	\draw (2,2) to (7,1.2);
	\draw (3,2) to (8,1.2);
	\draw (4,2) to (2,1.2);
	\draw plot [smooth] coordinates {(5,2) (4,1.6) (4,1.2)};
	\draw (6,2) to (3,1.2);
	\draw (7,2) to (5,1.2);
	\draw (8,2) to (6,1.2);
\end{tikzpicture}\]
and we note that in the lower part of this diagram, which represents the permutation $(2,7,5,4)(3,8,6)$, the strings associated to each $y_i$ do not
cross each other, which demonstrates that $(2,7,5,4)(3,8,6)$ is in $\mathcal{R}_\mu$.
Now in the upper part of the diagram, the arrangement of strings encodes the permutation $(2,3) \in S_3$ below both $y_1$ and $y_3$, while the strings
below $y_2$ encode the identity permutation in $S_2$. We remove the upper part of the diagram and let these permutations act on their respective
elements $y_i$, yielding
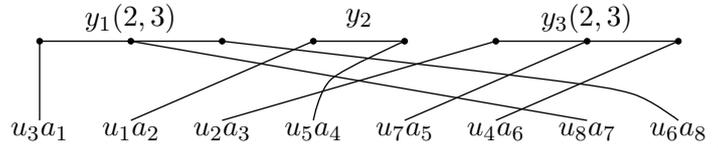
\[\begin{tikzpicture}[line width=0.5pt,xscale=1.2]
	%top row
	\node at (2,2.5){$y_1 (2,3)$};
	\node at (4.5,2.5){$y_2$};
	\node at (7,2.5){$y_3 (2,3)$};
	%lines and nodes for top row
	\draw (1,2.2) to (3,2.2);
		\node[dot] at (1,2.2){};
		\node[dot] at (2,2.2){};
		\node[dot] at (3,2.2){};
	\draw (4,2.2) to (5,2.2);
		\node[dot] at (4,2.2){};
		\node[dot] at (5,2.2){};
	\draw (6,2.2) to (8,2.2);
		\node[dot] at (6,2.2){};
		\node[dot] at (7,2.2){};
		\node[dot] at (8,2.2){};
	%bottom row
	\node at (1,1){$u_3 a_1$};
	\node at (2,1){$u_1 a_2$};
	\node at (3,1){$u_2 a_3$};
	\node at (4,1){$u_5 a_4$};
	\node at (5,1){$u_7 a_5$};
	\node at (6,1){$u_4 a_6$};
	\node at (7,1){$u_8 a_7$};
	\node at (8,1){$u_6 a_8$};
	%lines top to bottom
	\draw (1,2.2) to (1,1.15);
	\draw (2,2.2) to (7,1.15);
	\draw plot [smooth] coordinates {(3,2.2) (6.5,1.7) (7.5,1.5) (8,1.15)};
	\draw (4,2.2) to (2,1.15);
	\draw plot [smooth] coordinates {(5,2.2) (4.2,1.7) (4,1.15)};
	\draw (6,2.2) to (3,1.15);
	\draw (7,2.2) to (5,1.15);
	\draw (8,2.2) to (6,1.15);
\end{tikzpicture}.\]
Under our mapping, this corresponds to the pure tensor 
\begin{multline*}
  u_3 a_1 \otimes u_8 a_7 \otimes u_6 a_8 \otimes u_1 a_2 \otimes u_5 a_4 \otimes u_2 a_3 \otimes u_7 a_5 \otimes u_4 a_6 \otimes\\
  y_1(2,3) \otimes y_2 \otimes y_3(2,3) \otimes (2,7,5,4)(3,8,6),
\end{multline*}
and by letting $(x_1,x_2,x_3,x_4,x_5,x_6,x_7,x_8) = (u_3, u_6, u_8, u_1, u_5, u_2, u_4, u_7)$, $\sigma = (1,2,3)(4,6,8,7,5)$ and
$\gamma = (1,3,8,7,4)(2,6)$, and noting as above that then $\gamma\sigma = (2,8,5,4)(3,7,6) = (2,3)(7,8)\cdot(2,7,5,4)(3,8,6)$
where $(2,3)(7,8) \in S_\mu$ and $(2,7,5,4)(3,8,6) \in \mathcal{R}_\mu$,
we may verify that this is indeed the image of \eqref{diag_eg_elt_of_diag_module:eq} under the action of \eqref{diag_eg_elt_of_A_wr_Sn:eq}
as given by \eqref{wr_mod_full_action:eq}.
In the general case, for the $A{\wr}S_n$-module $\Theta^\mu(\ul{X},\ul{Y})$, let $d$ be the diagram formed from the permutation diagram of
$\gamma \in \mathcal{R}_\mu$ with labels $y_1$ to $y_r$ on the top row and labels $u_1$ to $u_n$ on the bottom row, and let $a$ be the element
$(\sigma;a_1,\ldots,a_n)$ of $A{\wr}S_n$. Then we have $\gamma\sigma = \theta\zeta$ where $\theta \in S_\mu$ and $\zeta \in \mathcal{R}_\mu$, and so
$\theta$ corresponds to some element $(\theta_1,\ldots,\theta_r)$ of $S_{\mu_1} \times\cdots\times S_{\mu_r}$ under the canonical isomorphism.
Then the image of $d$ under the action of $a$ is the diagram formed from the permutation diagram of $\zeta$ with top row labels
$y_1\theta_1$ to $y_r\theta_r$ and bottom row labels $u_{(1)\sigma^{-1}}a_1$ to $u_{(n)\sigma^{-1}}a_n$; we leave it to the reader to convince
themselves that in this diagram the nodes of the $i$\tss{th} block on the top row are connected to elements of $X_i$, and moreover that this diagram
does indeed represent the action of $a$ on the pure tensor of $\Theta^\mu(\ul{X},\ul{Y})$ represented by $d$.

In order to use our work below to obtain an alternative proof of the result of Geetha and Goodman that the wreath product of a cyclic cellular algebra
with $S_n$ is again cyclic cellular, we shall use the following.

\begin{proposition}\label{Theta_cyclic:prop}
If $X_1,\ldots,X_r$ are cyclic $A$-modules, and for each $i$, $Y_i$ is a cyclic $kS_{\mu_i}$-module, then $\Theta^\mu(\ul{X},\ul{Y})$ is a cyclic
$A{\wr}S_n$-module for any $r$ part composition $\mu$ of $n$. Indeed, if $x_i$ is a generator for $X_i$ and $y_i$ is a generator for $Y_i$, the diagram 
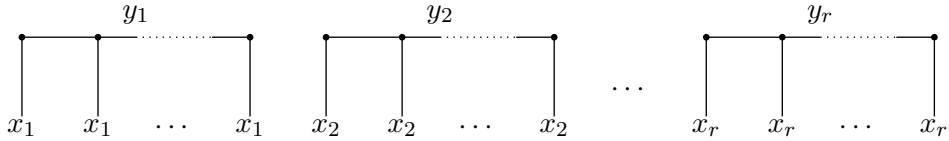
\[\begin{tikzpicture}[line width=0.5pt]
	%top row
	\node at (2.5,2.5){$y_1$};
	%lines and nodes for top row
	\draw (1,2.2) to (2.5,2.2);
	\draw[dotted] (2.5,2.2) to (3.5,2.2);
	\draw (3.5,2.2) to (4,2.2);
		\node[dot] at (1,2.2){};
		\node[dot] at (2,2.2){};
		\node[dot] at (4,2.2){};
	%bottom row
	\node at (1,1){$x_1$};
	\node at (2,1){$x_1$};
	\node at (3,1){$\cdots$};
	\node at (4,1){$x_1$};
	%lines top to bottom
	\draw (1,2.2) to (1,1.15);
	\draw (2,2.2) to (2,1.15);
	\draw (4,2.2) to (4,1.15);
	%%%%
	%top row
	\node at (6.5,2.5){$y_2$};
	%lines and nodes for top row
	\draw (5,2.2) to (6.5,2.2);
	\draw[dotted] (6.5,2.2) to (7.5,2.2);
	\draw (7.5,2.2) to (8,2.2);
		\node[dot] at (5,2.2){};
		\node[dot] at (6,2.2){};
		\node[dot] at (8,2.2){};
	%bottom row
	\node at (5,1){$x_2$};
	\node at (6,1){$x_2$};
	\node at (7,1){$\cdots$};
	\node at (8,1){$x_2$};
	%lines top to bottom
	\draw (5,2.2) to (5,1.15);
	\draw (6,2.2) to (6,1.15);
	\draw (8,2.2) to (8,1.15);
	%%%%
	\node at (9,1.5){$\cdots$};
	%%%%
	%top row
	\node at (11.5,2.5){$y_r$};
	%lines and nodes for top row
	\draw (10,2.2) to (11.5,2.2);
	\draw[dotted] (11.5,2.2) to (12.5,2.2);
	\draw (12.5,2.2) to (13,2.2);
		\node[dot] at (10,2.2){};
		\node[dot] at (11,2.2){};
		\node[dot] at (13,2.2){};
	%bottom row
	\node at (10,1){$x_r$};
	\node at (11,1){$x_r$};
	\node at (12,1){$\cdots$};
	\node at (13,1){$x_r$};
	%lines top to bottom
	\draw (10,2.2) to (10,1.15);
	\draw (11,2.2) to (11,1.15);
	\draw (13,2.2) to (13,1.15);
\end{tikzpicture}\]
(where each $x_i$ appears $\mu_i$ times) generates $\Theta^\mu(\ul{X},\ul{Y})$.
\end{proposition}
\begin{proof}
Let $d_0$ be the diagram in the proposition. It is easy to see that we may obtain any diagram in $\Theta^\mu(\ul{X},\ul{Y})$ by first applying an
element $(\theta;1,\ldots,1)$ of $A{\wr}S_n$, where $\theta \in S_\mu$, in order to replace each element $y_i$ in $d_0$ with an arbitrary element of
$Y_i$, then applying $(\gamma;1,\ldots,1)$ for some $\gamma \in \mathcal{R}_\mu$ to arrange the strings of the diagram, and finally applying an element
$(e;a_1,\ldots,a_n)$ to replace each element $x_i$ with an arbitrary element of $X_i$. Since $\Theta^\mu(\ul{X},\ul{Y})$ is spanned by diagrams,
the proof is complete.
\end{proof}

\section{The iterated inflation structure of the wreath product algebra}\label{it_inf_awrsn:sec}
Now we turn to the case where our interest lies. Let $A$ be a cellular algebra with anti-involution $\ast$ and cellular data $(\Lambda,M,C)$, where
$|\Lambda|=r$; since any partial order may be extended to a total order, we assume without loss of generality that $\Lambda$ is totally ordered, and so
we may list its elements in order as $\lambda_1 > \lambda_2 > \cdots > \lambda_r$.
We write $\Delta^\lambda$ for the right cell module associated to $\lambda \in \Lambda$ as noted above. For convenience we may omit the cell index
superscript from elements of the cellular basis, so we write $C_{S,T}$ rather than $C^\lambda_{S,T}$. We have a basis of $A{\wr}S_n$ consisting of
all elements of the form $(\sigma; C_{S_1,T_1},\ldots,C_{S_n,T_n})$ where $\sigma \in S_n$ and each $C_{S_i,T_i}$ is some element of
the cellular basis of $A$; note that we allow the elements $C_{S_i,T_i}$ to be associated to different cell indices. We shall denote this basis by
$\mathcal{A}$. Now elements of $\mathcal{A}$ are represented by diagrams like, for example,
\begin{equation}\label{unsplit_diag:diag}
\begin{tikzpicture}[line width=0.5pt,xscale=1.5]
	%top row
	\node[dot] at (1,2){};
	\node[dot] at (2,2){};
	\node[dot] at (3,2){};
	\node[dot] at (4,2){};
	\node[dot] at (5,2){};
	%bottom row
	\node at (1,1){$C_{S_1,T_1}$};
	\node at (2,1){$C_{S_2,T_2}$};
	\node at (3,1){$C_{S_3,T_3}$};
	\node at (4,1){$C_{S_4,T_4}$};
	\node at (5,1){$C_{S_5,T_5}$};
	%lines "top to bottom"
	\draw (1,2) to (3,1.15);
	\draw (2,2) to (1,1.15);
	\draw (3,2) to (5,1.15);
	\draw (4,2) to (2,1.15);
	\draw (5,2) to (4,1.15);
\end{tikzpicture}
\end{equation}
but we want a slightly different representation. Indeed, in the diagram \eqref{unsplit_diag:diag}, we replace each
$C_{S_i,T_i}$ with the pair $S_i,T_i$, and then move the $S_i$ up to the top of the associated string, to get
\[\begin{tikzpicture}[line width=0.5pt,xscale=1.5]
	%top row
	\node at (1,2){$S_3$};
	\node at (2,2){$S_1$};
	\node at (3,2){$S_5$};
	\node at (4,2){$S_2$};
	\node at (5,2){$S_4$};
	%bottom row
	\node at (1,1){$T_1$};
	\node at (2,1){$T_2$};
	\node at (3,1){$T_3$};
	\node at (4,1){$T_4$};
	\node at (5,1){$T_5$};
	%lines "top to bottom"
	\draw (1,1.8) to (3,1.2);
	\draw (2,1.8) to (1,1.2);
	\draw (3,1.8) to (5,1.2);
	\draw (4,1.8) to (2,1.2);
	\draw (5,1.8) to (4,1.2);
\end{tikzpicture}.\]
We thus obtain a different way of representing elements of $\mathcal{A}$, as diagrams of the form
\begin{equation}\label{split_diag_xy:diag}
\begin{tikzpicture}[line width=0.5pt,xscale=1.5]
	%top row
	\node at (1,2){$U_1$};
	\node at (2,2){$U_2$};
	\node at (3,2){$U_3$};
	\node at (4,2){$U_4$};
	\node at (5,2){$U_5$};
	%bottom row
	\node at (1,1){$W_1$};
	\node at (2,1){$W_2$};
	\node at (3,1){$W_3$};
	\node at (4,1){$W_4$};
	\node at (5,1){$W_5$};
	%lines "top to bottom"
	\draw (1,1.8) to (3,1.2);
	\draw (2,1.8) to (1,1.2);
	\draw (3,1.8) to (5,1.2);
	\draw (4,1.8) to (2,1.2);
	\draw (5,1.8) to (4,1.2);
\end{tikzpicture}
\end{equation}
consisting of a permutation diagram where the nodes on the top and bottom rows are replaced with elements
$U_i,W_i \in \sqcup_{\lambda \in \Lambda}M(\lambda)$, such that if $U_i$ on the top row is connected to $W_j$ on the bottom row, then we must have
$U_i,W_j \in M(\lambda)$ for some $\lambda \in \Lambda$ (i.e. $U_i$ and $W_j$ lie in \emph{the same} set $M(\lambda)$). Note that the diagram
\eqref{split_diag_xy:diag} represents the element
\[ \bigl((1,3,5,4,2);\,C_{U_2,W_1},\,C_{U_4,W_2},\,C_{U_1,W_3},\,C_{U_5,W_4},\,C_{U_3,W_5}\bigr) \in A{\wr}S_5. \]
Now given any such diagram, for each $i \in \{1,\ldots,r\}$ we let $\mu_i$ be the number of elements $U_j$ such that $U_j \in M(\lambda_i)$. We thus
obtain a composition $\mu = (\mu_1,\ldots,\mu_r)$ of $n$ (note that some of the parts $\mu_i$ may be zero in general). We call this the
\emph{layer index} of the diagram, and also of the element of $\mathcal{A}$ which it represents. We let $k\mathcal{A}_\mu$ be the $k$-span of all
elements of $\mathcal{A}$ with layer index $\mu$, and we let $I(n,r)$ be the set of all $r$-part compositions of $n$ with non-negative integer entries.
Then $A{\wr}S_n = \bigoplus_{\mu \in I(n,r)}k\mathcal{A}_\mu$. For a layer index $\mu$, we define a \textit{half diagram of type $\mu$} to be a tuple
$(U_1,\ldots,U_n)$ of $n$ elements  of $\sqcup_{\lambda \in \Lambda}M(\lambda)$, such that there are exactly $\mu_i$ elements of $M(\lambda_i)$ for
each $i$. We define $\mathcal{V}_\mu$ to be the set of all half diagrams of type $\mu$. Now if $(U_1,\ldots,U_n)$ is a half diagram of type $\mu$,
then we may easily see that there is a unique element $\epsilon$ of $\mathcal{R}_\mu$ such that
$(U_{(1)\epsilon},\ldots,U_{(n)\epsilon})$ lies in the set
$M(\lambda_1)^{\mu_1}{\times}\cdots{\times}M(\lambda_r)^{\mu_r}$; we shall call this $\epsilon$ the \textit{shape} of the
half diagram $(U_1,\ldots,U_n)$.

Let $E$ be the diagram with top row $U_1$ to $U_n$, bottom row $W_1$ to $W_n$ (reading from left to right), and where $\sigma \in S_n$ is the
permutation such that $U_i$ is connected to $W_{(i)\sigma}$; then $E$ represents the element
\[ \bigl( \sigma\,;\, C[U_{(1)\sigma^{-1}},W_1],\ldots,C[U_{(n)\sigma^{-1}},W_n] \bigr)\]
where to ease the notation we allow ourselves to write $C[U,W]$ for $C_{U,W}$.
Suppose $E$ has layer index $\mu$. We may decompose $E$ into three pieces of data, namely the half diagrams $(U_1,\ldots,U_n)$, $(W_1,\ldots,W_n)$ of
type $\mu$, formed from the top and bottom rows of $E$ respectively, and the element $(\pi_1,\ldots,\pi_r)$ of the group
$S_{\mu_i}\times\cdots\times S_{\mu_r}$ where $\pi_i \in S_{\mu_i}$ is such that (counting from
left to right) the $j$\tss{th} element of $M(\lambda_i)$ on the top row is connected to the $(j)\pi_i$\tss{th} element of $M(\lambda_i)$ on the bottom
row; thus $\pi_i$ records how the elements of $M(\lambda_i)$ on the top row are connected to the elements of $M(\lambda_i)$ on the bottom
row. For example, suppose that $r=3$ and that the diagram \eqref{split_diag_xy:diag} has layer index $(3,0,2)$ with $U_1,U_2,U_4 \in M(\lambda_1)$
and $U_3,U_5 \in M(\lambda_3)$. Then $(\pi_1,\pi_2,\pi_3) = \bigl((1,3,2),e,(1,2)\bigr)$ (note that $e$ here is the unique element of the trivial
group $S_{\mu_2} = S_0$). It is easy to see that if $\epsilon,\delta$ are the shapes of $(U_1,\ldots,U_n)$ and $(W_1,\ldots,W_n)$ respectively, and
further if $\pi$ is the image of $(\pi_1,\ldots,\pi_r)$ under the natural identification of $S_{\mu_i}\times\cdots\times S_{\mu_r}$ with the Young
subgroup $S_\mu$ of $S_n$, then $\sigma = \epsilon^{-1}\pi\delta$. If we now let
$V_\mu$ be the $k$-vector space with basis $\mathcal{V}_\mu$, then the above decomposition is easily seen to afford a $k$-linear bijection
\[V_\mu \otimes kS_\mu \otimes V_\mu \longrightarrow k\mathcal{A}_\mu\]
given by mapping
\[ (U_1,\ldots,U_n) \otimes \pi \otimes (W_1,\ldots,W_n), \]
to
\[\bigl(\epsilon^{-1}\pi\delta\,;\; C[U_{(1)(\epsilon^{-1}\pi\delta)^{-1}},W_1],\ldots,C[U_{(n)(\epsilon^{-1}\pi\delta)^{-1}},W_n]\bigr)\]
where $\epsilon$ is the shape of $(U_1,\ldots,U_n)$ and $\delta$ is the shape of $(W_1,\ldots,W_n)$. We equip the set $I(n,r)$ with the usual dominance
order on compositions. We thus have a decomposition $A{\wr}S_n = \bigoplus_{\mu \in I(n,r)}V_\mu \otimes kS_\mu \otimes V_\mu$. Now take
$\mathcal{V}_\mu$ as above, $B_\mu$ to be $kS_\mu$ and $\mathcal{B}_\mu$ to be $S_\mu$. We may easily see that our basis $\mathcal{A}$ is indeed the
basis of $A{\wr}S_n$ obtained from the bases $\mathcal{V}_\mu$ and $\mathcal{B}_\mu$ as in section \ref{it_inf:sec}, and we shall now prove that our
decomposition exhibits $A{\wr}S_n$ as an iterated inflation with respect to the anti-involution given by \eqref{anti_inv_def:eq} and the cellular
structure on the algebras $kS_\mu$ as in Proposition \ref{ten_prod_sn:prop}. Thus, we must prove that the equations \eqref{it_inf_anti_gen:eq} and
\eqref{it_inf_mult_gen:eq} hold. The fact that equation \eqref{it_inf_anti_gen:eq} holds follows easily from the description of the anti-involution on
$A{\wr}S_n$ given after equation \eqref{anti_inv_def:eq}. To prove that \eqref{it_inf_mult_gen:eq} holds, we shall prove the following slightly
stronger result.

\begin{proposition}\label{mult_cond_for_wrpr:prop}
Let $\mu \in I(n,r)$, and let $u=(U_1,\ldots,U_n), w=(W_1,\ldots,W_n)$ be elements of $\mathcal{V}_\mu$ and $\pi=(\pi_1,\ldots,\pi_r) \in S_\mu$ such
that the element of $\mathcal{A}$ corresponding to the pure tensor $u \otimes{\pi}\otimes w$ has layer index $\mu$. Further, let
$a = (\sigma;a_1,\ldots,a_n)$ be a pure tensor in $A{\wr}S_n$. Then we have
$(u \otimes {\pi} \otimes w) \cdot a \equiv u \otimes {\pi}\,\theta_\mu(w,a) \otimes \phi_\mu(w,a)$ modulo
elements of $\mathcal{A}$ of layer index strictly less than $\mu$, where $\theta_\mu(w,a) \in S_\mu$ and $\phi_\mu(w,a) \in V_\mu$ are independent of
$u$ and ${\pi}$.
\end{proposition}
Note that in the proposition we allow the $a$ in $\theta_\mu(w,a)$ and $\phi_\mu(w,a)$ to be any pure tensor in $A{\wr}S_n$ rather than just an element
of $\mathcal{A}$ as required in \eqref{it_inf_mult_gen:eq}.
\begin{proof}
Let $\epsilon,\delta \in \mathcal{R}_\mu$ be the shapes of $u$ and $w$ respectively, so that $u \otimes {\pi} \otimes w$ corresponds to the element
\[\bigl(\epsilon^{-1}\pi\delta\,;\; C[U_{(1)(\epsilon^{-1}\pi\delta)^{-1}},W_1],\ldots,C[U_{(n)(\epsilon^{-1}\pi\delta)^{-1}},W_n]\bigr).\]
Then
\begin{multline*}
(u \otimes {\pi} \otimes w)(\sigma;a_1,\ldots,a_n)=\\
\bigl(\epsilon^{-1}\pi\delta\,;\; C[U_{(1)(\epsilon^{-1}\pi\delta)^{-1}},W_1],\ldots,C[U_{(n)(\epsilon^{-1}\pi\delta)^{-1}},W_n]\bigr)
\bigl(\sigma;a_1,\ldots,a_n\bigr)=\\
\bigl(\epsilon^{-1}\pi\delta\sigma; C[U_{(1)(\epsilon^{-1}\pi\delta\sigma)^{-1}},W_{(1)\sigma^{-1}}]a_1,\ldots,
C[U_{(n)(\epsilon^{-1}\pi\delta\sigma)^{-1}},W_{(n)\sigma^{-1}}]a_n\bigr).
\end{multline*}
For each $i=1,\ldots,n$, we have by \eqref {cell_alg_mult_def:eq} that if $\lambda$ is the element of $\Lambda$ such that
$U_{(i)(\epsilon^{-1}\pi\delta\sigma)^{-1}},W_{(i)\sigma^{-1}} \in M(\lambda)$, then
\[ C[U_{(i)(\epsilon^{-1}\pi\delta\sigma)^{-1}},W_{(i)\sigma^{-1}}]a_i \equiv
\sum_{X_i \in M(\lambda)} R_{a_i}(W_{(i)\sigma^{-1}},X_i)C[U_{(i)(\epsilon^{-1}\pi\delta\sigma)^{-1}},X_i] \]
modulo cellular basis elements of lower cell index. Using this, we see that $(u \otimes {\pi} \otimes w)(\sigma;a_1,\ldots,a_n)$
is congruent modulo elements of $\mathcal{A}$ of lower layer index to
\begin{multline}\label{cong_equiv_prod_u_pi_v_sig_a:eq}
\sum_{X_1}\cdots\sum_{X_n}\left(\prod_{i=1}^n R_{a_i}\bigl(W_{(i)\sigma^{-1}},X_i\bigr)\right)
\bigl(\epsilon^{-1}\pi\delta\sigma\,;\; C[U_{(1)(\epsilon^{-1}\pi\delta\sigma)^{-1}},X_1],\ldots,\\
C[U_{(n)(\epsilon^{-1}\pi\delta\sigma)^{-1}},X_n]\bigr).
\end{multline}
Now $X_i$ lies in the same set $M(\lambda)$ as $W_{(i)\sigma^{-1}}$, and from this we may easily see that the shape of $(X_1,\ldots,X_n)$ is the unique
element $\zeta$ of $\mathcal{R}_\mu$ such that $\delta\sigma = \theta\zeta$ for $\theta \in S_\mu$. Thus in \eqref{cong_equiv_prod_u_pi_v_sig_a:eq} we
have
\begin{multline*}
\bigl(\epsilon^{-1}\pi\delta\sigma\,;\;
C[U_{(1)(\epsilon^{-1}\pi\delta\sigma)^{-1}},X_1],\ldots,C[U_{(n)(\epsilon^{-1}\pi\delta\sigma)^{-1}},X_n]\bigr)\\
=\bigl(\epsilon^{-1}\pi\theta\zeta\,;\;
C[U_{(1)(\epsilon^{-1}\pi\theta\zeta)^{-1}},X_1],\ldots,C[U_{(n)(\epsilon^{-1}\pi\theta\zeta)^{-1}},X_n]\bigr)
\end{multline*}
which we now see corresponds to the pure tensor $u \otimes \pi\theta \otimes (X_1,\ldots,X_n)$, and hence \eqref{cong_equiv_prod_u_pi_v_sig_a:eq} is
equal to
\[
u \otimes \pi\theta \otimes \left( \sum_{X_1}\cdots\sum_{X_n}\left(\prod_{i=1}^n R_{a_i}\bigl(W_{(i)\sigma^{-1}},X_i\bigr)\right)
(X_1,\ldots,X_n)\right).
\]
Thus, setting $\theta_\mu(w,a)$ to be the unique element $\theta$ of $S_\mu$ such that $\delta\sigma = \theta\zeta$ for
$\zeta \in \mathcal{R}_\mu$ and $\phi_\mu(w,a)$ to be
\begin{equation}\label{phi_mu_v_a_formula:eq}
\sum_{X_1}\cdots\sum_{X_n}\left(\prod_{i=1}^n R_{a_i}\bigl(W_{(i)\sigma^{-1}},X_i\bigr)\right)(X_1,\ldots,X_n),
\end{equation}
we see that $(u \otimes \pi \otimes w)(\sigma;a_1,\ldots,a_n) \equiv u \otimes \pi\theta_\mu(w,a) \otimes \phi_\mu(w,a)$
modulo lower layers, and furthermore these values depend only on $w$ and $a$, as required.
\end{proof}

By the results in Section \ref{it_inf:sec}, we now have that $A{\wr}S_n$ is a cellular algebra; further, we may use Proposition \ref{ten_prod_sn:prop}
to see that the set indexing the cell modules of $A{\wr}S_n$ is the set of all pairs $\bigl(\mu,(\nu_1,\ldots,\nu_r)\bigr)$ where $\mu$ is an
$r$-component composition $(\mu_1,\ldots,\mu_r)$ of $n$ (recalling that $r = |\Lambda|$), and $\nu_i$ is a partition of $\mu_i$. Thus in any such pair
we have $\mu = (|\nu_1|,\ldots,|\nu_r|)$, and so we lose no information if we omit the partition $\mu$ from these pairs. Hence we may identify the set
of cell indices of $A{\wr}S_n$ with the set of all $r$-tuples $(\nu_1,\ldots,\nu_r)$ of partitions such that $|\nu_1|+\cdots+|\nu_r| = n$
(with $\nu_i = ()$ allowed); such tuples are called \textit{multipartitions} of $n$ of \emph{length} $r$. We now give a statement of the
cellularity of $A{\wr}S_n$.

\begin{theorem}\label{wr_prod_cell:thm}
Let $A$ be a cellular algebra with anti-involution $\ast$ and poset $\Lambda$ of cell indices.
Let $\ul{\Lambda}^r_n$ denote the set of all \textit{multipartitions} of $n$ of length $r$. Then $A{\wr}S_n$ is a cellular algebra with respect to a
tuple of cellular data including the anti-involution given for $\sigma \in S_n$ and $a_1,\ldots,a_n \in A$ by
\[ (\sigma;a_1,\ldots,a_n)^\ast = \bigl(\sigma^{-1}\,;\,a_{(1)\sigma}^\ast,\ldots,a_{(n)\sigma}^\ast \bigr) \]
and also the poset consisting of $\ul{\Lambda}^r_n$ with the following partial order:
if $(\nu_1,\ldots,\nu_r),(\eta_1,\ldots,\eta_r) \in \ul{\Lambda}^r_n$ then $(\nu_1,\ldots,\nu_r) \geqslant (\eta_1,\ldots,\eta_r)$ means
either that  $(|\nu_1|,\ldots,|\nu_r|) \trianglerighteq (|\eta_1|,\ldots,|\eta_r|)$ or that 
$|\nu_i| = |\eta_i| \text{ and } \nu_i\trianglerighteq\eta_i \text{ for each $i$}$.
\end{theorem}

In the next section, we shall consider the cell modules which arise from this structure; in particular we shall recover the result of Geetha and
Goodman that if $A$ is cyclic cellular, then so is $A{\wr}S_n$.

We conclude this section by remarking that the most natural partial order on the poset $\ul{\Lambda}^r_n$ is the \emph{dominance order on
multipartitions} (see for example \cite[Definition 3.1, (1)]{GEGO}). We note that this dominance order is strictly stronger than the order we have
obtained on $\ul{\Lambda}^r_n$, and moreover that, subject to the assumption that $A$ is cyclic cellular, Geetha and Goodman obtained the dominance
order in their cellularity result. It seems natural to suppose that in the general case $A{\wr}S_n$ should still be cellular with respect to the
cellular data from
Theorem \ref{wr_prod_cell:thm} if we replace the above partial order on $\ul{\Lambda}^r_n$ with the dominance order; however, it would be impossible to
prove this result using the method of iterated inflations due to the structure of the partial orders obtained via this method. If it is indeed possible
to obtain the dominance order in our result, it would be necessary to use a more refined combinatorial argument of the kind used by Geetha
and Goodman in \cite{GEGO}.

\section{The cell and simple modules of the wreath product algebra}
Recall that the cell modules $\Delta^{\lambda_i}$ of $A$ are indexed by the cell indices ${\lambda_1 > \lambda_2 > \cdots > \lambda_r}$. In the
sequel we shall also allow ourselves to write $\Delta^{\lambda_i}$ as $\Delta(\lambda_i)$ when this makes our formulae more readable.
We shall now consider the cell modules of $A{\wr}S_n$. We know that these are indexed by length $r$ multipartitions of $n$; let $(\nu_1,\ldots,\nu_r)$
be such a multipartition and $\mu$ the composition $(|\nu_1|,\ldots,|\nu_r|)$, so that $\mu_i = |\nu_i|$. We shall show that the cell module
$\Delta^{(\nu_1,\ldots,\nu_r)}$ is isomorphic to the module
$\Theta^\mu\bigl((\Delta^{\lambda_1},\ldots,\Delta^{\lambda_r}),(S^{\nu_1},\ldots,S^{\nu_r})\bigr)$.

Now we know from Proposition \ref{ten_prod_sn:prop} and the results in section \ref{it_inf:sec} that, as a $k$-vector space,
$\Delta^{(\nu_1,\ldots,\nu_r)}$ may naturally be identified with
\begin{equation}\label{cell_mod_vecsp_ident:eq}
S^{\nu_1}\otimes\cdots\otimes S^{\nu_r} \otimes V_\mu,
\end{equation}
so let us consider the structure of the vector space $V_\mu$. Indeed, let $\alpha_1,\ldots\alpha_n$ be elements of $\Lambda$ such that
\[
 (\alpha_1,\ldots,\alpha_n) =
 (\underbrace{\lambda_1,\lambda_1,\ldots,\lambda_1}_{\text{$\mu_1$ places}},\underbrace{\lambda_2,\ldots,\lambda_2}_{\text{$\mu_2$ places}},
 \lambda_3,\ldots,\underbrace{\lambda_r,\ldots,\lambda_r}_{\text{$\mu_r$ places}}).
\]
Let $(X_1,\ldots,X_n)$ be a half diagram  in $\mathcal{V}_\mu$. Then the shape of $(X_1,\ldots,X_n)$ is the unique element $\gamma$ of
$\mathcal{R}_\mu$ such that $(X_1,\ldots,X_n)$ lies in $M(\alpha_{(1)\gamma^{-1}})\times\cdots\times M(\alpha_{(n)\gamma^{-1}})$. We now see that
\[\mathcal{V}_\mu = \bigsqcup_{\gamma \in \mathcal{R}_\mu} M(\alpha_{(1)\gamma^{-1}})\times\cdots\times M(\alpha_{(n)\gamma^{-1}})\]
and hence if we identify the half diagram $(X_1,\ldots,X_n)$ with the pure tensor $C_{X_{1}}\otimes\cdots\otimes C_{X_{n}}$, we obtain a natural
identification of $k$-vector spaces
\begin{equation}\label{v_mu_id:eq}
 V_\mu = \bigoplus_{\gamma \in \mathcal{R}_\mu} \Delta(\alpha_{(1)\gamma^{-1}})\otimes\cdots\otimes \Delta(\alpha_{(n)\gamma^{-1}}).
\end{equation}
We shall henceforth consider these two vector
spaces to be thus identified; further, we shall abuse terminology and use the term \textit{pure tensor in $V_\mu$} to mean any pure tensor in any of the
summands in the right hand side of \eqref{v_mu_id:eq}. For example, using \eqref{cell_mod_def:eq}, we can show easily using 
\eqref{phi_mu_v_a_formula:eq} that under the identification \eqref{v_mu_id:eq} we have
\begin{equation}\label{phi_mu_v_a_formula_pure_ten:eq}
\phi_\mu\bigl(C_{W_1}\otimes\cdots\otimes C_{W_n},(\sigma;a_1,\ldots,a_n)\bigr) =
C_{W_{(1)\sigma^{-1}}}\,a_1\otimes \cdots \otimes C_{W_{(n)\sigma^{-1}}}\,a_n.
\end{equation}
In light of \eqref{cell_mod_vecsp_ident:eq}, we shall further speak of a \textit{pure tensor in $\Delta^{(\nu_1,\ldots,\nu_r)}$} to mean any pure
tensor of the form
\[ w_1\otimes\cdots\otimes w_r\otimes u_1\otimes \cdots \otimes u_n,\]
where $w_i \in S^{\nu_i}$ and $u_1\otimes \cdots \otimes u_n$ is a pure tensor in $V_\mu$. Using \eqref{phi_mu_v_a_formula_pure_ten:eq}
and the expression for $\theta_\mu(w,a)$ given near the end of the proof of Proposition \ref{mult_cond_for_wrpr:prop}, we may now verify that
the map taking the pure tensor
\[x_1 \otimes \cdots \otimes x_n \otimes y_1 \otimes \cdots \otimes y_r \otimes \gamma\]
in $\Theta^\mu\bigl((\Delta^{\lambda_1},\ldots,\Delta^{\lambda_r}),(S^{\nu_1},\ldots,S^{\nu_r})\bigr)$ (where $\gamma \in \mathcal{R}_\mu$)
to the pure tensor
\[y_1 \otimes \cdots \otimes y_r \otimes x_{(1)\gamma^{-1}} \otimes \cdots \otimes x_{(n)\gamma^{-1}}\]
in $\Delta^{(\nu_1,\ldots,\nu_r)}$ is an isomorphism of $A{\wr}S_n$-modules (but note that in order to apply the formula given in section
\ref{it_inf:sec} for the action of an iterated inflation on its cell modules, the arguments $w$ and $a$ in $\theta_\mu(w,a)$ and $\phi_\mu(w,a)$ must
be elements of the bases $\mathcal{A}$ and $\mathcal{V}_\mu$, respectively). We may now use Proposition \ref{Theta_cyclic:prop} and the fact that
all Specht modules are cyclic to obtain the following result.

\begin{proposition}
(Geetha and Goodman, \cite{GEGO}) If $A$ is cyclic cellular then so is $A{\wr}S_n$.
\end{proposition}

Now by equation \eqref{mult_in_a_layer:eq}, we know that the multiplication within each layer of $A{\wr}S_n$ is determined by a bilinear form,
$\psi_\mu$. Let $(U_1,\ldots,U_n),(W_1,\ldots,W_n)$ be half diagrams in $\mathcal{V}_\mu$, so that
$u = C_{U_1} \otimes \cdots \otimes C_{U_n}$ and $w = C_{W_1} \otimes\cdots\otimes C_{W_n}$ are pure tensors in $V_\mu$.
Now by equation \eqref{mult_in_a_layer:eq},
\begin{equation}\label{calc_of_bilin_form_first:eq}
 (u \otimes e \otimes u)(w \otimes e \otimes w) \equiv u \otimes \psi_\mu(u,w) \otimes w
\end{equation}
modulo lower layers. The element $u \otimes e \otimes u$ of $A{\wr}S_n$ is represented by the diagram
\[\begin{tikzpicture}[line width=0.5pt]
	%top row
	\node at (1,2){$U_1$};
	\node at (2,2){$U_2$};
	\node at (3,2){$\cdots$};
	\node at (4,2){$U_n$};
	%bottom row
	\node at (1,1){$U_1$};
	\node at (2,1){$U_2$};
	\node at (3,1){$\cdots$};
	\node at (4,1){$U_n$};
	%lines "top to bottom"
	\draw (1,1.8) to (1,1.2);
	\draw (2,1.8) to (2,1.2);
	\draw (4,1.8) to (4,1.2);

	\node at (5.5,1.5){$=$};

	%top row
	\node[dot] at (7,2){};
	\node[dot] at (8.1,2){};
	\node at (9.2,2){$\cdots$};
	\node[dot] at (10.3,2){};
	%bottom row
	\node at (7,1){$C_{U_1,U_1}$};
	\node at (8.1,1){$C_{U_2,U_2}$};
	\node at (9.2,1){$\cdots$};
	\node at (10.3,1){$C_{U_n,U_n}$};
	%lines "top to bottom"
	\draw (7,1.8) to (7,1.2);
	\draw (8.1,1.8) to (8.1,1.2);
	\draw (10.3,1.8) to (10.3,1.2);

\end{tikzpicture}\]
and of course the element $w \otimes e \otimes w$ is represented by a diagram which is the same except that each $U$ is replaced with a $W$. Thus we
find by concatenating and simplifying these diagrams that the product $(u \otimes e \otimes u)(w \otimes e \otimes w)$ corresponds to
\begin{equation}\label{bilin_form_diag_CUU_CVV}\begin{tikzpicture}[line width=0.5pt,xscale=2]
	%top row
	\node[dot] at (1,2){};
	\node[dot] at (2,2){};
	\node at (3,2){$\cdots$};
	\node[dot] at (4,2){};
	%bottom row
	\node at (1,1){$C_{U_1,U_1}C_{W_1,W_1}$};
	\node at (2,1){\hspace{2em}$C_{U_2,U_2}C_{W_2,W_2}$};
	\node at (3,1){$\cdots$};
	\node at (4,1){$C_{U_n,U_n}C_{W_n,W_n}$};
	%lines "top to bottom"
	\draw (1,1.8) to (1,1.2);
	\draw (2,1.8) to (2,1.2);
	\draw (4,1.8) to (4,1.2);
\end{tikzpicture}.\end{equation}
We may expand each of the products $C_{U_i,U_i}C_{W_i,W_i}$ in terms of the cellular basis of $A$ and use these expansions to write
\eqref{bilin_form_diag_CUU_CVV} as a linear combination of diagrams of the form
\[\begin{tikzpicture}[line width=0.5pt]
	%top row
	\node[dot] at (1,2){};
	\node[dot] at (2,2){};
	\node at (3,2){$\,\cdots$};
	\node[dot] at (4,2){};
	%bottom row
	\node at (1,1){$C_{X_1,Y_1}$};
	\node at (2,1){\hspace{1em}$C_{X_2,Y_2}$};
	\node at (3,1){$\,\cdots$};
	\node at (4,1){$C_{X_n,Y_n}$};
	%lines "top to bottom"
	\draw (1,1.8) to (1,1.2);
	\draw (2,1.8) to (2,1.2);
	\draw (4,1.8) to (4,1.2);
\end{tikzpicture}.\]
Further, from the fact that any product $C^\delta_{X,Y}C^\epsilon_{U,W}$ of cellular basis elements is a linear combination of elements
$C^\theta_{S,T}$ for $\theta$ equal to or less than both $\delta$ and $\epsilon$, we may easily see that all such diagrams have layer index
at most $\mu$, and moreover if for any $i$ we have that  $U_i$ and $W_i$ do not lie in the same set $M(\lambda)$, then all of the diagrams in the
expansion have layer index strictly less than $\mu$, and hence by \eqref{calc_of_bilin_form_first:eq} we see that we must have $\psi_\mu(u,w) = 0$ in
this case. Suppose now that $U_i$ and $W_i$ do indeed lie in the same set $M(\lambda)$ for each $i$; by (2.4.1) in \cite{GLCA}, we know that
$C_{U_i,U_i}C_{W_i,W_i}$ is congruent to $\langle C_{U_i},C_{W_i} \rangle C_{U_i,W_i}$ modulo cellular basis elements of lower cell index, where
$\langle \cdot,\cdot \rangle$ is the appropriate cell form. Thus we see that \eqref{bilin_form_diag_CUU_CVV} is congruent modulo lower layers to
\[\begin{tikzpicture}[line width=0.5pt]
	\node at (-2.2,1.5) {$\langle C_{U_1},C_{W_1} \rangle\langle C_{U_2},C_{W_2} \rangle \cdots \langle C_{U_n},C_{W_n} \rangle$};
	%top row
	\node at (1,2){$U_1$};
	\node at (2,2){$U_2$};
	\node at (3,2){$\cdots$};
	\node at (4,2){$U_n$};
	%bottom row
	\node at (1,1){$W_1$};
	\node at (2,1){$W_2$};
	\node at (3,1){$\cdots$};
	\node at (4,1){$W_n$};
	%lines "top to bottom"
	\draw (1,1.8) to (1,1.2);
	\draw (2,1.8) to (2,1.2);
	\draw (4,1.8) to (4,1.2);
\end{tikzpicture},\]
which represents the element
$\langle C_{U_1},C_{W_1} \rangle\langle C_{U_2},C_{W_2} \rangle \cdots \langle C_{U_n},C_{W_n} \rangle \, u \otimes e \otimes w$, and hence we find
that in this case
\[\psi_\mu(u,w) = \langle C_{U_1},C_{W_1} \rangle\langle C_{U_2},C_{W_2} \rangle \cdots \langle C_{U_n},C_{W_n} \rangle.\]
Note in particular that $\psi_\mu$ is thus in all cases $k$-valued. We can now use these values for $\psi_\mu$, together with equation
\eqref{bilin_form:eq} and Proposition \ref{ten_prod_sn:prop} to compute the values of the cell form on the cell module
$\Delta^{(\nu_1,\ldots,\nu_r)}$; indeed, if $y_1\otimes\cdots\otimes y_r\otimes u_1\otimes \cdots \otimes u_n$ and
$z_1\otimes\cdots\otimes z_r\otimes w_1\otimes \cdots \otimes w_n$
are pure tensors in the cell module $\Delta^{(\nu_1,\ldots,\nu_r)}$, then we see that
\begin{multline}\label{cell_form_within_omega}
\langle y_1\otimes\cdots\otimes y_r\otimes u_1\otimes \cdots \otimes u_n,\;z_1\otimes\cdots\otimes z_r\otimes w_1\otimes \cdots \otimes w_n \rangle =\\
\langle y_1,z_1 \rangle \cdots  \langle y_r,z_r \rangle  \langle u_1,w_1 \rangle \cdots  \langle u_n,w_n \rangle
\end{multline}
if $u_i$ and $w_i$ lie in the same $\Delta(\lambda)$ for each $i=1,\ldots,n$, and
\begin{equation}\label{cell_form_between_omegas}
\langle y_1\otimes\cdots\otimes y_r\otimes u_1\otimes \cdots \otimes u_n,\;z_1\otimes\cdots\otimes z_r\otimes w_1\otimes \cdots \otimes w_n \rangle = 0
\end{equation}
otherwise.

Next we seek to describe the cell radical of $\Delta^{(\nu_1,\ldots,\nu_r)}$. Using \eqref{cell_mod_vecsp_ident:eq} and
\eqref{v_mu_id:eq}, we have isomorphisms of $k$-vector spaces
\begin{align}\begin{split}\label{vsp_ident_cell_ulnu:eq}
\Delta^{(\nu_1,\ldots,\nu_r)} &\cong S^{\nu_1}\otimes\cdots\otimes S^{\nu_r} \otimes V_\mu\\
&\cong \bigoplus_{\gamma \in \mathcal{R}_\mu} S^{\nu_1}\otimes\cdots\otimes S^{\nu_r} \otimes \Delta(\alpha_{(1)\gamma^{-1}})\otimes\cdots\otimes
\Delta(\alpha_{(n)\gamma^{-1}}).
\end{split}\end{align}
For $\gamma \in \mathcal{R}_\mu$, let
$\Omega_\gamma = S^{\nu_1}\otimes\cdots\otimes S^{\nu_r} \otimes \Delta(\alpha_{(1)\gamma^{-1}})\otimes\cdots\otimes
\Delta(\alpha_{(n)\gamma^{-1}})$. Now we see from \eqref{cell_form_between_omegas} that if $\gamma,\beta$
are distinct elements of $\mathcal{R}_\mu$ and $u \in \Omega_\gamma$, $w \in \Omega_\beta$ then $\langle u,w \rangle = 0$. It follows that, if we let
$R_\gamma$ be the radical of the restriction to $\Omega_\gamma$ of $\langle \cdot,\cdot \rangle$, then the cell radical of
$\Delta^{(\nu_1,\ldots,\nu_r)}$ is $\bigoplus_{\gamma \in \mathcal{R}_\mu}R_\gamma$.

Let us fix a basis in each $\Delta^\lambda$
and each $S^\nu$; from these bases we obtain a basis of pure tensors in each $\Omega_\gamma$. Let $G_{\nu_i}$ be the Gram matrix of the
cell form of $S^{\nu_i}$ and $G_{\alpha_i}$ be the Gram matrix of the cell form of $\Delta^{\alpha_i}$, with respect to our chosen bases. 
If we let $B_\gamma$ be the Gram matrix of the restriction of the cell form to $\Omega_\gamma$ with respect to our basis, then we see by
\eqref{cell_form_within_omega} that $B_\gamma$ is the matrix Kronecker product
$G_{\nu_1}\otimes\cdots\otimes G_{\nu_r}\otimes G_{\alpha_{(1)\gamma^{-1}}}\otimes\cdots\otimes G_{\alpha_{(n)\gamma^{-1}}}$.
By fixing some total order on the set $R_\gamma$ and concatenating our bases of the $\Omega_\gamma$ in this order, we obtain a basis of
$\Delta^{(\nu_1,\ldots,\nu_r)}$; using \eqref{cell_form_between_omegas}, we see that its Gram matrix with
respect to this basis is of block diagonal form with diagonal blocks $B_\gamma$ for $\gamma \in \mathcal{R}_\mu$. From this we see (using the fact that
the rank of the Kronecker product of two matrices is the product of their ranks) that the rank of the cell form on $\Delta^{(\nu_1,\ldots,\nu_r)}$
is $|\mathcal{R}_\mu|$ times the product of the ranks of the cell forms of the cell modules
$S^{\nu_1},\ldots,S^{\nu_r},\Delta^{\alpha_1},\ldots,\Delta^{\alpha_n}$.

Now in constructing the above basis of pure tensors for $\Delta^{(\nu_1,\ldots,\nu_r)}$ as above, we may choose our basis of each cell module of $A$
and $kS_n$ by taking a basis of the cell radical and extending this to a basis of the whole cell module. If we do this, then we see that an element
$y_1\otimes\cdots\otimes y_r\otimes u_1\otimes \cdots \otimes u_n$ of the basis of pure tensors for $\Delta^{(\nu_1,\ldots,\nu_r)}$ must lie in the
cell radical if any $y_i$ or $u_i$ is an element of the cell radical of the cell module in which it lies. By the above calculation of the rank of the
cell form on $\Delta^{(\nu_1,\ldots,\nu_r)}$, we see that the number of such elements must be equal to the dimension of the cell radical, and so we
have now found a basis of the cell radical inside a basis of the whole cell module.

We can now use the theory of cellular algebras from section 3 of \cite{GLCA} together with our basis of $\Delta^{(\nu_1,\ldots,\nu_r)}$
to deduce some results about the simple modules $L^{(\nu_1,\ldots,\nu_r)}$ and semisimplicity of $A{\wr}S_n$. These results are already known for
wreath products $A{\wr}S_n$
with $A$ a general (i.e. not cellular) algebra given extra assumptions on the field (see for example \cite[Lemma 3.4]{CHTAN}), and in particular for
the case $k\bigl(G{\wr}S_n\bigr) \cong (kG){\wr}S_n$ where $G$ is a finite group (see for example Chapter 4 of \cite{JAMKER} for the case where
the field is algebraically closed). However, if $A$ is cellular then our work shows that these results hold \emph{with no restriction on the field at
all}. Given the importance of cellular algebras in certain areas of representation theory we are confident that they will prove useful.

Recall that $\Lambda_0$ indexes the simple modules of $A$. Let $\bigl(\ul{\Lambda}^r_n\bigr)_0$
denote the set of elements $(\nu_1,\ldots,\nu_r) \in \ul{\Lambda}^r_n$ such that the cell radical of $\Delta^{(\nu_1,\ldots,\nu_r)}$ is a proper
submodule of $\Delta^{(\nu_1,\ldots,\nu_r)}$, so that $\bigl(\ul{\Lambda}^r_n\bigr)_0$ indexes the simple modules of $A{\wr}S_n$. Recall that our
field $k$ has characteristic $p$, which may be zero or a prime.

\begin{theorem}\label{label_simples:thm} The set $\bigl(\ul{\Lambda}^r_n\bigr)_0$ indexing the simple modules of $A{\wr}S_n$ consists exactly of those
$(\nu_1,\ldots,\nu_r) \in \ul{\Lambda}^r_n$ such that $\nu_i = ()$ whenever $\lambda_i \in \Lambda\setminus\Lambda_0$ and all $\nu_i$
are $p$-restricted (recall that $()$ is $p$-restricted for any $p$).
\end{theorem}

In light of Theorem \ref{label_simples:thm}, we see that if we let $s$ be the number of simple modules of $A$ and we let
$\hat\lambda_1 > \hat\lambda_2 > \cdots > \hat\lambda_s$ be the elements of $\Lambda_0$, then the simples of $A{\wr}S_n$ may in fact be indexed by the
set $\ul{\Lambda}^s_n(p)$ consisting of all length $s$ multipartitions of $n$ with $p$-restricted entries (compare \cite[Proposition 3.7]{CHTAN}).

\begin{theorem}Let $(\nu_1,\ldots,\nu_r) \in \bigl(\ul{\Lambda}^r_n\bigr)_0$. Then corresponding to the isomorphism
\eqref{vsp_ident_cell_ulnu:eq}, we have an isomorphism of $k$-vector spaces
\[
L^{(\nu_1,\ldots,\nu_r)} \cong \bigoplus_{\gamma \in \mathcal{R}_\mu} D^{\nu_1}\otimes\cdots\otimes D^{\nu_r} \otimes
L^{\alpha_{(1)\gamma^{-1}}}\otimes\cdots\otimes L^{\alpha_{(n)\gamma^{-1}}}.
\]
Moreover, $L^{(\nu_1,\ldots,\nu_r)}$ has a representation by diagrams of the form \eqref{elt_of_Theta_mu_X_Y:diag} in exactly the same way as
$\Delta^{(\nu_1,\ldots,\nu_r)}$, by simply using elements of $D^{\nu_i}$ rather than $S^{\nu_i}$ and elements of  $L^{\alpha_i}$ rather
than $\Delta^{\alpha_i}$; the action on such diagrams is exactly the same as described above. We thus see that $L^{(\nu_1,\ldots,\nu_r)}$ is isomorphic
as an $A{\wr}S_n$-module to $\Theta^{\mu}\bigl((L^{\lambda^1},\ldots,L^{\lambda^r}),(D^{\nu_1},\ldots,D^{\nu_r})\bigr)$, where
$\mu = (|\nu_1|,\ldots,|\nu_r|)$ (a composition of $n$), and for convenience we let $L^\lambda=0$ for $\lambda \in \Lambda\setminus\Lambda_0$.
\end{theorem}

We thus see that if we index the simples by $\ul{\Lambda}^s_n(p)$ as above, then the simple indexed by $(\hat\nu_1,\ldots,\hat\nu_s)$
(where each $\hat\nu_i$ is thus a $p$-restricted partition) is isomorphic to
$\Theta^{\hat\mu}\bigl((\Delta^{\hat\lambda^1},\ldots,\Delta^{\hat\lambda^s}),(S^{\hat\nu_1},\ldots,S^{\hat\nu_s})\bigr)$,
where $\hat\mu = (|\hat\nu_1|,\ldots,|\hat\nu_s|)$.

\begin{theorem}Let $(\nu_1,\ldots,\nu_r) \in \bigl(\ul{\Lambda}^r_n\bigr)_0$. Then we have
$L^{(\nu_1,\ldots,\nu_r)} \cong \Delta^{(\nu_1,\ldots,\nu_r)}$
if and only if $D^{\nu_i} \cong S^{\nu_i}$ for each $i=1,\ldots,r$ and whenever we have $\nu_i \neq ()$ we have
$L^{\lambda_i} \cong \Delta^{\lambda_i}$.
\end{theorem}

Our final result is a criterion for semisimplicity; compare \cite[Lemma 3.5]{CHTAN}. 

\begin{theorem}If $A$ is a cellular algebra, then $A{\wr}S_n$ is semisimple if and only if both $kS_n$ and $A$ are semisimple.
\end{theorem}

\end{document}